\documentclass[a4paper,11pt]{article}
\usepackage[ngerman, english]{babel}
\usepackage[T1]{fontenc}
\usepackage[utf8]{inputenc}
\usepackage{amsmath}
\usepackage{amssymb}
\usepackage{amsthm}
\usepackage{graphicx}
\usepackage{here}
\usepackage{color}
\usepackage{xcolor}
\usepackage{enumerate}
\usepackage{lmodern}
\usepackage{fancyvrb}
\usepackage[plainpages=false]{hyperref}
\usepackage{caption}
\usepackage{subfigure}
\usepackage{epstopdf}
\usepackage{verbatim}
\usepackage{tikz}
\newcommand{\ca}{\mathcal}

\newtheorem{defi}{Definition}[section]
\newtheorem{theo}[defi]{Theorem}
\newtheorem{prop}[defi]{Proposition}
\newtheorem{lem}[defi]{Lemma}
\newtheorem{con}[defi]{Conjecture}
\newtheorem{obs}[defi]{Observation}
\newtheorem{cor}[defi]{Corollary}

\theoremstyle{remark}

\usepackage[normalem]{ulem}

\newcommand{\adjtwo}[0]{indirectly adjacent}

\usepackage{authoraftertitle}

\usepackage{authblk}

\title{Edge-coloring $4$- and $5$-regular projective planar graphs with no Petersen-minor}
	\author[1]{Arnott Kidner\thanks{Funded by the Deutsche Forschungsgemeinschaft (DFG, German Research Foundation) – 546530571}}
    
    \author[1]{Eckhard Steffen}
    \author[1,2]{Weiqiang Yu}
	\affil[1]{\scriptsize{Department of Mathematics, Paderborn University, Germany}}
    \affil[2]{\scriptsize{Department of Mathematics, Zhejiang Normal University, China}}
    \affil[ ]{\footnotesize{kidner@mail.upb.de, es@upb.de, wyu@mail.upb.de}}
    \date{}

\begin{document}
	
	\maketitle
	
	\begin{abstract}
		An $r$-regular graph is an $r$-graph, if every odd set of vertices is connected to its complement by at least $r$ edges. We prove for $r \in \{4,5\}$, every projective planar $r$-graph with no Petersen-minor is $r$-edge colorable.  
	\end{abstract}
	

	\section{Introduction} \label{Section: Intro}
	All graphs considered in this paper are finite and may have parallel edges 
	but no loops. A graph is simple if it contains no loops and no parallel edges. The underlying simple graph of a graph $G$ is denoted by $G_s$. 
	
	The vertex set of a graph $G$ is denoted by $V(G)$ and its edge set by $E(G)$. 
	For two disjoint sets $X,Y\subseteq V(G)$ we denote by $E_G(X,Y)$ the set of edges with one end in $X$ and the other end in $Y$. 
	If $X=\{x\}$ and $Y=\{y\}$, then we write $E_G(x,y)$ instead of $E_G(\{x\},\{y\})$.
	If $\vert E_G(x,y) \vert=k$, then we say that $xy$ is a $k${\em -edge} and set
	$\mu(G)=\max\{\vert E_G(u,v) \vert \colon u,v\in V(G)\}$.  
	If $Y= V(G) \setminus X$ we denote $E_G(X,Y)$ by $\partial_G(X)$. In this case, $\partial_G(X)$ is a {\em cut}.    
    Further, if $\vert X \vert$
	is odd, then $\partial_G(X)$ is called an {\em odd cut}.
    A graph is \emph{$r$-regular} if every vertex has degree $r$.
    A \emph{circuit} is a connected 2-regular graph. If $C$ is a circuit with vertices $v_1, \dots, v_n$ and edges $v_1v_2, \dots,v_nv_1$, then we also write $C=v_1\dots v_n$. An $r$-regular graph $G$ is an \emph{$r$-graph}, if $|\partial_G(X)| \geq r$ for every odd cut $\partial_G(X)$. 
	The graph induced by $X$ is denoted by $G[X]$. 
	
	A set of $k$ vertices, whose removal increases the number of components of a graph, is called a {\em $k$-vertex cut}.
	A graph $G$ on at least $k+1$ vertices is called {\em $k$-connected} if $G-X$ is connected for every $X\subset V(G)$ with $|X|\leq k-1$. Moreover, the graph obtained from $G$ by identifying all vertices in $X$ to a single vertex $w_X$ and deleting all resulting loops is denoted by $G/X$. If $G[X]$ is connected, we say $G/X$ is obtained from $G$ by contracting $X$.
    A graph $H$ is a \emph{minor} of the graph $G$ if $H$ can be obtained from a subgraph of $G$ by contracting edges. We also say that $G$ has an $H$-minor. 
    If $H=P$, where $P$ denotes the Petersen graph, then we also say that $G$ has a {\em Petersen-minor} or a $P$-minor for short. 
    
    Let $\Psi(G)$ be an embedding of $G$ in a surface and $f$ be a face with facial walk $W(f)$. The length of 
	$W(f)$ is denoted by $d(f)$ and called the 
	{\em size} of $f$. A face of size $k$ is also called a {\em $k$-face} and 
	a face of size at least $k$ is called a {\em $k^+$-face}.
	If $x$ is a vertex or an edge of the facial walk of $f$, then we also say that $x$ is a vertex or an edge of $f$.
	For further definitions on graphs on surfaces we refer to \cite{Mohar_Thomassen_Book}.
	
	Let $T \subseteq V(G)$ be of even cardinality. 
	A \emph{$T$-join} is a set $M$ of edges such that $T$ is the set of vertices of odd degree in $G[M]$. 	
	A \emph{matching} is a set $M\subseteq E(G)$ such that no two edges of $M$ are adjacent. If every vertex of $G$ is incident with an edge of $M$, then $M$ is a \emph{perfect matching}.
	Let $S$ be a set. An \emph{edge coloring} of $G$ is a mapping $f\colon E(G)\to S$. It is a \emph{$k$-edge coloring} if $|S|=k$, and it is \emph{proper} if $f(e) \not = f(e')$ for any two adjacent edges $e$ and $e'$. The smallest integer $k$ for which $G$ admits a proper $k$-edge coloring is the \emph{edge-chromatic number} of $G$, which is denoted by $\chi'(G)$. If $\chi'(G)$ equals the maximum degree of $G$, then $G$ is \emph{class $1$}; otherwise $G$ is \emph{class $2$}. For a proper edge coloring $c$ and a set of edges $E_G(x,y) = \{e_1, \dots, e_k\}$, we define $c(E_G(x,y)) = \{c(e_1), \dots, c(e_k)\}$.
	
	Tait \cite{Tait_1880} proved that the 4-Color-Theorem is equivalent to the statement that every planar 3-graph is class 1. 
	Tutte's  4-flow conjecture \cite{tutte1966algebraic} states that every 
	bridgeless graph with no Petersen-minor has a nowhere-zero 4-flow. For cubic graphs this conjecture is equivalent to the statement that every $3$-graph with no Petersen-minor is class $1$. 
	Seymour stated the following conjecture for planar $r$-graphs.

	\begin{con}[\cite{seymour1979unsolved}] \label{Conj: Seymour exact}
		For $r \geq 1$, every planar $r$-graph is class $1$. 
	\end{con}

	The following conjecture, which naturally generalizes the edge coloring version of Tutte's $4$-flow conjecture for cubic graphs to $r$-graphs, is attributed to Seymour in \cite{guenin2003packing}.
	
	\begin{con} \label{Conj: Petersen minor}
		Let $G$ be an $r$-graph. If $G$ has no Petersen-minor, then $G$ is class 1. 
	\end{con}
	
	Conjecture~\ref{Conj: Seymour exact} is proved for $r \leq 8$ in a series of papers \cite{7chudnovsky2015edge, 8chudnovsky2015edge, dvovrak2016packing, guenin2003packing, NaserasrYu-2023}. 
	Since planar graphs have no Petersen-minor Conjecture~\ref{Conj: Petersen minor} is true for planar $r$-graphs if $r \leq 8$. 

Recently, Inoue, Kawarabayashi, Miyashita, Mohar, and Sonobe \cite{inoue2024t_cubic_projective_planar} proved far reaching extensions of the 4-Color-Theorem to projective planar graphs.
	A {\em snark} is a cubic graph of class 2 with girth at least $5$ and cyclic connectivity at least 4. 

\begin{theo}[\cite{inoue2024t_cubic_projective_planar}] \label{Thm: Mohar_etal projective plane}
		The only snark embeddable in the projective plane is the Petersen graph. 
	\end{theo}
    
    Let $G, H$ be two disjoint $r$-graphs.
	For $v \in V(G), u \in V(H)$, an $r${\em -sum} 
	of $G$ and $H$ is an $r$-graph that is obtained from $G-v$ and $H-u$
	by adding $r$ edges joining vertices of degree smaller than $r$ of $G-v$ with those of $H-u$. A $3$-graph is $P$\emph{-like} if it can be obtained from the Petersen graph by repeatedly building 3-sums with planar 3-graphs. 
    The following theorem then gives a complete classification of all projective planar 3-graphs which are class 2. 

\begin{theo}[\cite{inoue2024t_cubic_projective_planar}] \label{Thm: Mohar_etal P-like} A projective planar 3-graph is class 2 if and only if it is $P$-like. 
\end{theo}
    
	In particular, Theorem~\ref{Thm: Mohar_etal P-like} implies
    that projective planar $3$-graphs without a Petersen-minor are class 1 graphs. The following theorem is the main result of this paper. 
    
	\begin{theo} \label{Thm: main}
		Let $r \leq 5$ and $G$ be a projective planar $r$-graph. 
		If $G$ has no Petersen-minor, then $G$ is class $1$. 
	\end{theo}

Theorem~\ref{Thm: main} will be proved inductively. The case $r=5$ relies on the case $r=4$,
which relies on Theorem~\ref{Thm: Mohar_etal projective plane}. 
The next section gives some basic facts and structural properties 
of possible minimum counterexamples which apply to both cases $r=4$ and $r=5$. 
Section~\ref{Sec: Proof main Thm 4} gives a proof for the case $r=4$ and Section~\ref{Sec: Proof main Thm 5} for the case $r=5$. The paper concludes with some 
final remarks in Section~\ref{Sec: concluding remarks}.
    

	\subsection{Preliminaries}
	
Let $G$ be an $r$-graph. An odd cut $\partial_G(X)$ is a \emph{non-trivial tight cut} if 
$\vert \partial_G(X) \vert = r$ and $\vert X \vert, \vert V(G) \setminus X \vert > 1$. 	The
following lemma is proved by estimates of the number of edges incident to a 2-vertex cut.
In what follows indices are added in the cyclic group of appropriate order.

\begin{lem} [\cite{ma2024_equivalences}]  \label{Lemma-2-vertex-cut}
	For each $ r\geq 3 $, if an $r$-graph $ G $ has a $2$-vertex cut, then either $ G $ has a non-trivial tight cut or $ G_s $ is a circuit of length $ 4 $.
\end{lem}

For a positive even integer $n$ let $C_n = v_1 \dots v_n$ be a circuit of length $n$.
A \emph{$C_n$-swap} of $G$ is the graph that is obtained from 
$G - \{v_{2i}v_{2i+1} \colon i \in [n/2]\}$ 
by adding an edge $\overline{v_{2i-1}v_{2i}}$ for each $i \in [n/2]$. We also say that $G'$ is a $v_1 \dots v_n$-swap of $G$.

\begin{lem} \label{Lemma: swap}
	Let $r \geq 3$, $H$ be a connected simple graph, and let $G$ be an $r$-graph with no $H$-minor and no non-trivial tight cut. If $G'$ is a $C_n$-swap of $G$ for $n \in \{4,6\}$, then $G'$ is an $r$-graph with no $H$-minor. 
	Furthermore, if $G$ is embeddable in some surface $S$, then $G'$ is embeddable in $S$. 
\end{lem}

\begin{proof} Since $G_s'$ is a subgraph of $G_s$, every simple minor of
	$G'$ is a minor of $G$. Hence, $G'$ has no $H$-minor. For $n = 4$ we
	have that $\vert \partial_{G'}(X) \vert \geq \vert \partial_{G}(X) \vert - 2 
	\geq r$. For $n = 6$ it follows that
	$\vert \partial_{G'}(X) \vert \geq \vert \partial_{G}(X) \vert - 3$.
	Thus, by parity $\vert \partial_{G'}(X) \vert \geq \vert \partial_{G}(X) \vert - 2 \geq r$ and hence, $G'$ is an $r$-graph. 
	The $C_n$-swap applied on the embedded graph keeps the embedding. 
\end{proof}

\subsection{Minimum counterexamples} \label{Sec: Pre}
 
We will apply the following definition. 
	
	\begin{defi}\label{def:smaller}
		Let $H$ and $G$ be $r$-graphs. Then 
	 $H$ is \emph{smaller} than $G$ if 
		\begin{enumerate}
			\item $\vert V(H)\vert < \vert V(G)\vert$, or if 
			\item $\vert V(H)\vert = \vert V(G)\vert$, and $H$ has more $3$-edges, or if
			\item $\vert V(H)\vert = \vert V(G)\vert$, $G$ and $H$ have the same number of $3$-edges, and $H$ has more $2$-edges.
		\end{enumerate}
	\end{defi}

Suppose to the contrary that Theorem~\ref{Thm: main} is not true. For $r \in \{4,5\}$ let $H$ be a minimum counterexample with regard to Definition~\ref{def:smaller}. 
We assume that $H$ is not planar, since planar $4$- and $5$-graphs are class $1$, see \cite{guenin2003packing}. 

Let $\Psi(H)$ be an embedding of $H$ into the projective plane.
Since $H$ is not planar and
$r$-graphs are $2$-connected, $\Psi(H)$ has representativity at least $2$. Hence, all facial walks are circuits (in $H$), see e.g. Proposition 5.5.11 \cite{Mohar_Thomassen_Book}. Thus, we can assume that $\Psi(H)$ is a circular 2-cell embedding of $H$ in the projective plane.
Two faces of $\Psi(H)$ are {\em adjacent}, when their facial circuits 
intersect in edges, and they are {\em incident}, when their facial circuits intersect in vertices.

\begin{theo} \label{Thm: min counterexample} $H$ has the following properties: 
\begin{enumerate}
	\item $H$ has no non-trivial tight cut,
	\item $H$ is $3$-connected and thus $\delta(H_s) \geq 3$,
    \item $\vert \partial_H(X) \vert \geq 4$ for every $X \subset V(G)$,
	\item $\mu(H) \leq r-2$.
\end{enumerate}
\end{theo}
	
\begin{proof} 
		1.: Suppose to the contrary that $H$ has a non-trivial tight cut $\partial_H(Y)$.
		Then $H[Y]$ and $H[V(H)\setminus Y]$ are connected. Thus $H \slash Y$ and 
		$H \slash (V(H)\setminus Y)$ are $P$-minor-free $r$-graphs, which 
		are embeddable in the projective plane. Thus, by the minimality of $H$, they are class $1$ and it follows that 
		they have $r$-edge colorings which can be combined to an $r$-edge coloring of $H$, a contradiction. Hence, $H$ has no non-trivial tight cut. 
		
		2 and 3.: If $H$ would have a $2$-vertex cut, then, by
		Lemma~\ref{Lemma-2-vertex-cut}, $H_s$ is a circuit of length $4$. Hence, $H$ is $r$-edge colorable, a contradiction.
		Consequently, $H$ is $3$-connected and $\vert \partial_H(X) \vert \geq 4$ for every $X \subset V(G)$.

	4.: If $\mu(H) \geq r-1$, then $H$ is either the graph on two vertices which are connected by $r$ edges or
	$H$ has a 2-cut, a contradiction to 3. 
\end{proof}

Theorem~\ref{Thm: main} will be proved by using the discharging method. The 2- and 3-faces are the only faces which will receive charge in the discharging process. The following 
statement follows from the fact that the facial walks in $\Psi(H)$ are circuits. 
	
\begin{obs} \label{Obs: 2- and 3-faces}
	For $k \in \{2,3\}$, if $f$ is a $k$-face in $\Psi(H)$, then $f$ is adjacent to $k$ pairwise different faces. 
\end{obs}

	A multiset $M_1, \dots, M_r$ of $r$ $T$-joins is an \emph{$e$-coloring}
	for $e \in E(G)$ if $M_i \cap M_j \subseteq \{e\}$ for
	$i \not = j \in [r]$ and there is no multiset of $r$ $T$-joins $M_1', \dots,M_r'$ which
	pairwise intersect at most in $e$ and 
	$\vert \{M_i' \colon e \in M_i', i \in [r]\}\vert < \vert \{M_i \colon e \in M_i, i \in [r]\}\vert$. Furthermore, an $e$-coloring $M_1, \dots, M_r$ is \emph{strong} if 
	$\vert \{M_i \colon e \in M_i, i \in [r]\}\vert \leq 3$. 
	
	Let $e \in E(G)$ and $M_1, \dots,M_r$ be an $e$-coloring. For $i \in [r]$, 
	an odd cut $\partial_G(X)$ is a \emph{mate of $M_i$}, if $e \in \partial_G(X)$ and 
	$\vert \partial_G(X) \cap M_j \vert = 1$ for
	all $j \in [r] \setminus \{i\}$.
	The following lemma is a customized version of Lemma 2.6 of \cite{guenin2003packing}. We include the proof for the sake of self containment. 
	
	\begin{lem} \label{Lemma: e-coloring Mate}
		For $e \in E(H)$, 
		if $M_1, \dots ,M_r$ is a strong $e$-coloring, then each of $M_1, \dots,M_r$ has a mate.
	\end{lem}
	
	\begin{proof} Let $e = xy$ and $l_e = \vert \{M_i \colon e \in M_i, i \in [r]\}\vert$. 
		If $l_e < 3$, then $l_e=1$ and $H$ is class $1$, a contradiction. Thus $l_e=3$ and all but one or
		two of $M_1, \dots ,M_r$ are perfect matchings of $H$. If $M_i$ is not a perfect matching,
		then for one of $x,y$, say for $x$, $\vert \partial_H(x) \cap M_i \vert = 3$ and 
		$\vert \partial_H(x) \cap M_j \vert = 1$ for each $j \not = i$. Thus, $\partial_H(x)$
		is a mate for $M_i$. 
		
		Now assume that $M_1$ is not a perfect matching, $M_i$ is a perfect matching of $H$ and 
		let $H' = H - M_i$.
		For $X \subset V(H)$, $\partial_H(X)$ is an odd cut in $H$ if and only if 
		$\partial_{H'}(X)$ is an odd cut in $H'$. If $\vert \partial_{H'}(Z) \vert \geq r-1$
		for each odd $Z \subset V(H)$, then it follows by induction that $H$ is class $1$, a contradiction. Thus,
		there is an odd cut $\partial_{H'}(Y)$ with $\vert \partial_{H'}(Y) \vert < r-1$.
        If $e\notin M_i$, then $\vert \partial_{H'}(Y) \vert \geq r-3$
		since every perfect matching intersects $\partial_{H'}(Y)$. However, $\partial_{H'}(Y) \cap M_1 \not = \emptyset$ and therefore,
		$e \in \partial_{H'}(Y)$ and $\partial_H(Y)$ is a mate for $M_i$.
        Now assume $e\in M_i$. Without loss of generality, let $e\in M_t, M_s$. Then $\vert\partial_{H'}(Y) \cap M_l\vert=1$ for each $l\in [r]\setminus\{i,t,s\}$. Therefore, $\partial_H(Y)$ is a mate for $M_i$.
	\end{proof}

Let $M_1, \dots, M_r$ be an $e$-coloring
	of $G$, $e=uv$, $\{u,u_1,u_2\} \subset V(G)$ be the vertex set of a triangle $T$, $v \not \in \{u_1,u_2\}$, and let $\partial_G(X_i)$ be a mate for $M_i$. We call $T$ a {\em bad} triangle, if $\vert E(T) \cap \partial_G(X_i) \cap M_i \vert \geq 2$.

\begin{lem} [\cite{guenin2003packing}] \label{Lemma: no bad triangle} If $M_1, \dots, M_r$ is an $e$-coloring of $H$, then for all $i \in [r]$,
 $M_i$ has a mate with no bad triangles.
	\end{lem}

For a face $f$ let $val(f)=\vert \{e \colon 
\exists f' \not = f \text{ such that } e \in E(W(f)) \cap E(W(f')) 
\text{ and } d(f') \geq 4 \} \vert$.

\begin{lem} \label{Lemma: (r-2)-edge}
Let $f$ be a $3^+$-face in $\Psi(H)$, and $u,v \in V(W(f))$. 
If $\vert E_H(u,v) \vert = r-2$, then $d(f) \geq r+1$ and $val(f) \geq r$. 
\end{lem}

\begin{proof} 
We first show that there is a strong $e$-coloring for each $e \in E_H(u,v)$. 
Let $H'$ be the
graph obtained from $H - E_H(u,v)$ by suppressing
$u$ and $v$. Then $H'$ has no $P$-minor and, by Theorem~\ref{Thm: min counterexample}, $H'$ is an $r$-graph. By induction hypothesis, $H'$ has an $r$-edge coloring, which easily can be modified to a strong $e$-coloring of $H$ for each $e \in E_H(u,v) = \{e_1, \dots,e_{r-2}\}$.
	
Let $N_H(u)=\{v, u_1, u_2\}$, $N_H(v)=\{u, v_1,v_2\}$, and $u_i, v_i$ be in the same face for $i \in \{1,2\}$. Since $\delta_s(H) \geq 3$, $u_1\neq u_2$ and $v_1\neq v_2$. 
There is a strong $e_1$-coloring, such that $e_1$ is contained in $M_1, M_2,M_3$, $e_i$ is contained in  $M_{i+2}$ for $i \in \{2, \dots, r-2\}$, and each of $u$ and $v$ is incident to either three edges of $M_1$ or three edges of $M_2$ and $M_3, \dots,M_r$ are perfect matchings of $H$. 
		
By Lemma~\ref{Lemma: e-coloring Mate}, there exists a mate $\partial_H(X_i)$ for $M_i$ for each $i \in [r]$. 
Then $E_H(u,v) \subseteq \partial_H(X_i)$ and hence,
$f$ contains at least $r-2$ edges of $\partial_H(X_i)$ for $i \in \{3, \dots,r\}$. Thus, $d(f)  \geq r+1$. 
By Lemma~\ref{Lemma: no bad triangle}, the face sharing an edge of each $M_1, \dots, M_r$ with $f$ contains another edge of $M_1, \dots, M_r$, which is not in $E_H(u,v)$. Thus, these faces are of size $4$ at least and therefore,
$val(f) \geq r$.
	\end{proof}

	
	\section{$4$-graphs} \label{Sec: Proof main Thm 4}
	
	In this section we will prove the following theorem. 
	
	\begin{theo} \label{Thm: main 4-graph}
		Let $G$ be a projective planar $4$-graph. 
		If $G$ has no Petersen-minor, then $G$ is class $1$.     
	\end{theo}
	
	Suppose to the contrary that Theorem~\ref{Thm: main 4-graph} is not true. 
	In the following, let $G$ be a minimum counterexample with regard to Definition~\ref{def:smaller}. 
	As in the previous section, we assume that $G$ is not planar and that $\Psi(G)$ 
	is a circular 2-cell embedding of $G$
	in the projective plane. 
	By Theorem~\ref{Thm: min counterexample}, a
	3-face is not adjacent to a 2-face. 

\subsection{Structure of faces}

    \begin{lem} \label{Lemma: e-coloring triangle}
        Let $f = vv_1v_2$ be a 3-face and for $i \in \{1,2\}$ let $f_i$ and be the face which is adjacent to $f$ by sharing the edge $vv_i$.
If $\vert N_G(v) \vert = 3$, then $d(f_i) \geq 5$ and $val(f_i) \geq 4$.
If $\vert N_G(v) \vert = 4$, then $d(f_i) \geq 5$ and $val(f_i) \geq 2$,
and there are two $4^+$-faces $f_i^1, f_i^2$ and $e_i^j \in W(f_i) \cap W(f_i^j)$ $(j \in \{1,2\})$
such that $e_i^1$ and $e_i^2$ are not adjacent to $vv_i$.
    \end{lem}

    \begin{proof}    
    Let $\partial_G(V(f)) = \{e_1, \dots,e_6\}$,
    with $e_i = v_1x_i$ for $i \in \{1,2\}$, $e_i = vx_i$ for $i \in \{3,4\}$, and $e_i = v_2x_i$ for $i \in \{5,6\}$.

    If $\vert N_G(v) \vert = 3$, then $x_3 = x_4$ and hence, $\vert E_G(v,x_4) \vert = 2$. Therefore,  $d(f_i) \geq 5$ and $val(f_i) \geq 4$, by Lemma~\ref{Lemma: (r-2)-edge}.

    Let $x_3 \not = x_4$, and  
    $z$ be the vertex which $V(f)$ is contracted to in $G/V(f)$ and let $e_i' = zx_i$. 
    Let $G' = ((G/V(f)) - \{ e_3',e_4' \} )+ x_3x_4$.
    Let $X,Y$ be odd sets in $G$ where $X \subseteq V(G) \setminus V(f)$ and $Y \subset V(G)$ with $V(f) \subseteq Y$. Let $X^-$ and $Y^-$ be the 
    corresponding sets in $G/V(f)$, where $z \in Y^-$. By Theorem~\ref{Thm: min counterexample},
    $6 \leq \vert \partial_G(X) \vert = \vert \partial_{G/V(f)}(X^-)\vert$ and 
    $6 \leq \vert \partial_G(Y) \vert = \vert \partial_{G/V(f)}(Y^-) \vert$. Hence,  
    $\vert \partial_{G'}(Z) \vert \geq 4$ for any odd set $Z \subset V(G')$ and 
    $G'$ is a projective planar 4-graph.
    The graph $G'$ can also be obtained from $G$ by minor operations and
    therefore, $G'$ has no $P$-minor.

    By the induction hypothesis, $G'$ is class 1.
    Let $e \in \{vv_1,vv_2 \}$ and $e_2,e_3\in f_1$.
    Without loss of generality, suppose $e = vv_1$.
    Any $4$-edge coloring of $G'$ extends to a strong $e$-coloring of $G$, such that there are two perfect matchings which do not contain the edges $v_1v_2$, $e_2$ and $e_3$. Furthermore, the edges $vv_1$ and $vv_2$ get all four colors. Thus by Lemma~\ref{Lemma: e-coloring Mate}, $f_1$ is adjacent to two $4^+$-faces $f_1^1, f_1^2$ and $e_1^j \in W(f_1) \cap W(f_1^j)$ $(j \in \{1,2\})$ such that $e_1^1$ and $e_1^2$ are not adjacent to $vv_1$. By symmetry, we have the same conclusion for $f_2$.
    \end{proof}

    \begin{figure}[ht]%
	\centering
	
	\includegraphics[height=0.4\textwidth]{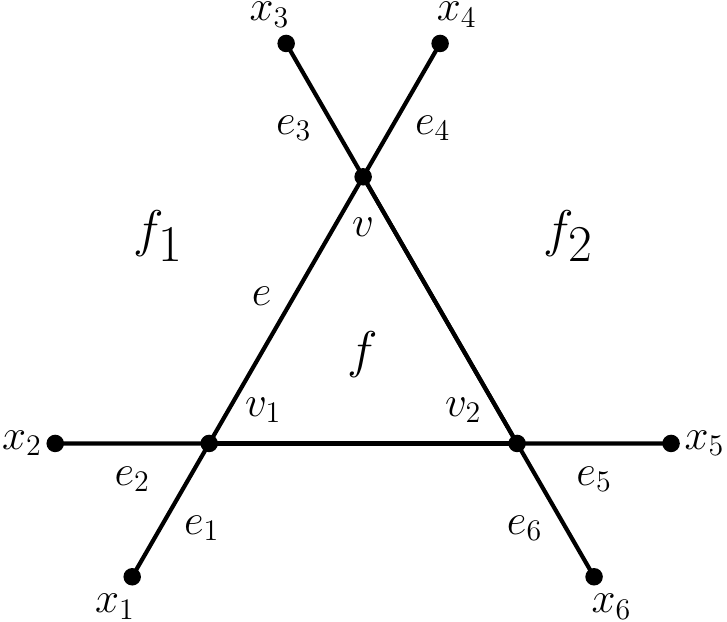}

	\caption{Lemma~\ref{Lemma: e-coloring triangle} if $\vert N_G(v) \vert = 4$.
    }\label{fig: e-coloring triangle}
\end{figure}

    \begin{cor}\label{cor:33 and 53}
       A $3$-face is not adjacent to a $3$- or a $4$-face. A $5$-face is adjacent to at most two $3$-faces.
    \end{cor}

\subsection{Discharging Procedure}
    
	We may now employ the discharging technique to finish the proof. 
    Assign to each face $f$ of $\Psi(G)$
the initial charge of $d(f) - 4$, which is denoted by $\omega(f)$. As $G$ is a $4$-graph, Euler's formula implies that $|F(G)|=\frac{|E(G)|}{2}+1$, where $F(G)$ denotes the set of faces of $\Psi(G)$. Therefore, 

\begin{equation} \label{Eq: total initial charge 4}
\sum_{f\in F(G)}\omega(f)=2|E(G)|-4|F(G)|=-4.
\end{equation} 

Next we show that  charge can be redistributed according to the following rules, such that 
each face has non-negative charge, in contradiction to equation~\ref{Eq: total initial charge 4}.

    \begin{description}\setlength{\itemsep}{0mm}\setlength{\parsep}{0mm}
    \item[R1] Each $2$-face receives charge 1 from each of its adjacent face;
    \item[R2] Each $3$-face receives charge $\frac{1}{3}$ from each of its adjacent face.
    \end{description}

    The final charge of a face $f$ after redistribution of charge is denoted by $\omega'(f)$.

	By Lemma~\ref{Lemma: (r-2)-edge} and Corollary~\ref{cor:33 and 53}, every face, which is adjacent to a face of size at most 3, has size at least 5. By rules R1 and R2, $2$- and $3$-faces only receive charge and therefore, they have final charge of at least $0$. 
	
	Faces of size 4 are not affected by the discharging procedure, and thus, they have a final charge of $0$. 
	
	For a $5$-face $f$, if $f$ is adjacent to a $2$-face, then by Lemma~\ref{Lemma: (r-2)-edge}, $f$ is adjacent to exactly one $2$-face and no $3$-face. Thus, the charge of $f$ is reduced by 
    at most $1$ by rule R1. If $f$ is not adjacent to a $2$-face, then by Corollary~\ref{cor:33 and 53}, $f$ is adjacent to at most two $3$-faces. 
    Then, by R2, the charge of $f$ is reduced by at most $\frac{2}{3}$. Therefore, $\omega'(f) \geq 0$ in any case. 
	
	For a $6$-face $f$, if $f$ is not adjacent to any $2$-face, then 
    only rule R2 applies and $f$ loses at most $6\times\frac{1}{3}=2$ of charge. Otherwise, by Lemma~\ref{Lemma: (r-2)-edge}, $val(f)\geq 4$, $f$ loses at most $2$ of charge. Thus every $6$-face has a final charge of at least $0$.
	
	For a $7^+$-face $f$, by Theorem~\ref{Thm: min counterexample} and Lemma~\ref{Lemma: (r-2)-edge}, $f$ loses at most $\lfloor\frac{d(f)}{2}\rfloor+\frac{1}{3}(\lceil\frac{d(f)}{2}\rceil-4) \leq d(f)-4$ of charge. Thus, $f$ has final charge of at least $0$.

    Thus, $\sum_{f\in F(G)}\omega'(f) \geq 0$, a contradiction. This completes 
    the proof of Theorem~\ref{Thm: main 4-graph}.
	
	
	\section{$5$-graphs} \label{Sec: Proof main Thm 5}
	
	In this section we will prove the following theorem. 
	
	\begin{theo} \label{Thm: main 5-graph}
		Let $G$ be a projective planar $5$-graph. 
		If $G$ has no Petersen-minor, then $G$ is class $1$.     
	\end{theo}
	
	Suppose to the contrary that Theorem~\ref{Thm: main 5-graph} is not true. 
	In the following, let $G$ be a minimum counterexample with regard to Definition~\ref{def:smaller}. 
	As in Section~\ref{Sec: Pre} we assume that $G$ is not planar and that $\Psi(G)$ 
	is a circular 2-cell embedding of $G$ in the projective plane. 
	
	By Theorem~\ref{Thm: min counterexample}, a 3-face is adjacent to at most one 2-face. 
	For $k \in \{2,3,4\}$, a $k$-face that is adjacent to a 2-face is called a \emph{heavy} $k$-face.  
	Let $f$ and $f'$ be the two distinct faces which are adjacent to a 2-face.
	If $f$ is a heavy 3-face, then we say that $f'$ is \emph{\adjtwo} to $f$.

\subsection{Structure of faces}

Let $P_s$ be a path of length $m$ in $G_s$.
If $|E_G(v,w)| = 2$ for every edge $vw$ of $P_s$, then 
we say that $P$ is a \emph{2-face $m$-path} in $\Psi(G)$.
Furthermore, a $2$-face of $P$ is \emph{internal} if it is not incident to an end vertex of $P$.

\begin{lem} \label{Lemma: three 2-edges in a row}
    Let $P$ be a 2-face $3$-path and $f$ be the internal
    2-face with adjacent $3^+$-faces $f_1$ and $f_2$. For $i \in \{1,2\}$, let $n_i$ be the number of $4^+$-faces which are adjacent to $f_i$ and not incident to $f$.
    Then,
    
    \begin{enumerate} 
    \item $n_i \geq 1$ and $n_1 + n_2 \geq 4$.
    \item $f_i$ is adjacent to a $4^+$-face $f'_i$ such that one of their common edges is not incident to any vertex of $P$.
    \item $d(f_i) \geq 4$.
    \end{enumerate}
\end{lem}

\begin{proof} 
    Let $V(P_s) = \{v_1, \dots,v_4\}$, $E(P_s)= \{v_1v_2,v_2v_3,v_3v_4\}$, and $E_G(v_2,v_3) = \{e_1,e_2\}$.
    Let $G'$ be obtained from $G$ by deleting $e_2$ and contracting $E_G(v_1,v_2)$ and $E_G(v_3,v_4)$.
    Then, $G'$ is projective planar and has no $P$-minor.
    By Theorem~\ref{Thm: min counterexample}, $G'$ is a $5$-graph and by induction hypothesis, $G'$ has a $5$-edge coloring, which easily can be modified to a strong $e_1$-coloring of $G$. 
    We may assume $e_1$ is contained in $M_1, M_2,M_3$, $e_2$ is contained in $M_4$, and $M_3, M_4, M_5$ are perfect matchings of $G$. 

    By Lemma~\ref{Lemma: e-coloring Mate}, there exists a mate $\partial_G(X_t)$ for $M_t$ for all $t \in [5]$. We have $e_1,e_2\in \partial(X_t)$, since $e_1, e_2$ are parallel edges. 
      Thus, both $f_1$ and $f_2$ contain at least one edge of $\partial_G(X_5)$, say $e'_1$ and $e'_2$ respectively. 
    Further, for $i \in \{1,2\}$, the face sharing edge $e'_i$ with $f_i$ contains another edge of $M_5$, which is not $e_1$ or $e_2$. 
      Thus, these faces are of size at least $4$ and they cannot be incident with $v_2v_3$, which implies $n_i \geq 1$. 
      As both $f_1$ and $f_2$ contain at least one edge of $\partial_G(X_j)$ for $j \in \{3,4\}$, 
      at most one of $e'_1$ and $e'_2$ belongs to $\partial_G(X_j)$. 
      Further, together $f_1$ and $f_2$ contain in total at least two edges of $\partial_G(X_j)$, say $e_3$ and $e_4$.

      The face sharing $e_3$ and $e_4$ with $f_1$ and $f_2$ contains another edge of $M_j$ which is not $e_1$ or $e_2$.
      Thus, these faces are of size at least $4$ and they cannot be incident with $v_2v_3$. Thus, $n_1 + n_2 \geq4$. 
    
    Finally we observe that in the $e_1$-coloring, one of the parallel edges of $v_1v_2$ and $v_3v_4$ receive color $M_5$.
      Hence, $e'_1$ and $e'_2$ are not incident to $v_1$ and $v_4$.
      In particular, $d(f_i) \geq 4$.
\end{proof}

\begin{figure}[ht]%
	\centering
	
	\includegraphics[height=0.15\textwidth]{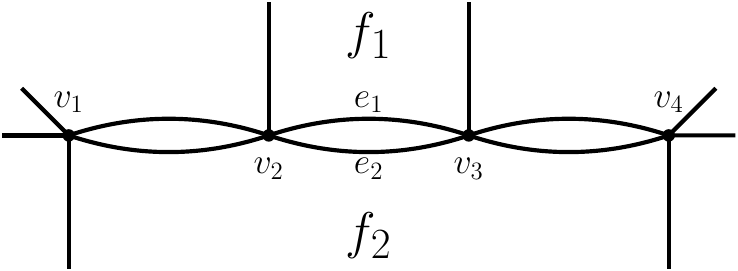}\hspace{3em}
	\includegraphics[height=0.15\textwidth]{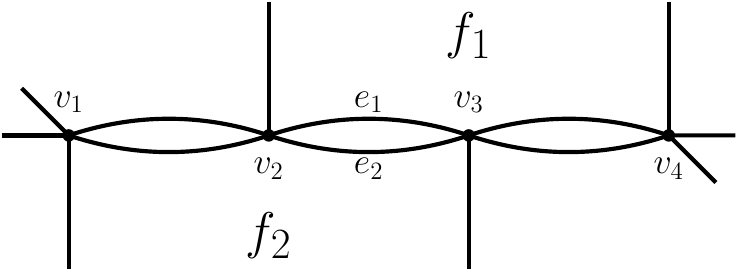}

	\caption{The two non-isomorphic cases for Lemma~\ref{Lemma: three 2-edges in a row}.
    }\label{fig: three 2-edges in a row}
\end{figure}

\begin{lem} \label{Lemma: e-coloring fat triangle}
	Let $f,f'$ be the two distinct faces which are adjacent to a 2-face with vertex set $\{v,w\}$.
	Let $f$ be a heavy 3-face and $f''$ be a $3^+$-face which is adjacent to $f$ and $v \in V(W(f''))$.
	If $\vert N_G(v) \vert = 4$, then $val(f')\geq 3$, $d_G(f') \geq 6$,
	and $val(f'') \geq 3$, $d_G(f'') \geq 5$.
\end{lem}
	
\begin{proof} Let $V(f) = \{u,v,w\}$,
	Let $\partial_G(V(f)) = \{e_1, \dots,e_7\}$,
	with $e_i = ux_i$ for $i \in \{1,2,3\}$, $e_i = vx_i$ for $i \in \{4,5\}$, and $e_i = wx_i$ for $i \in \{6,7\}$.
	Then $x_4 \not = x_5$, since $\vert N_G(v) \vert = 4$.
	
	Let $z$ be the vertex to which $V(f)$ is contracted in $G/V(f)$ 
	and $e_i'=zx_i$.
	Let $G' = ((G/V(f)) - \{e_4',e_5'\}) + x_4x_5$.

    Let $X,Y$ be odd sets in $G$ where $X \subseteq V(G) \setminus V(f)$ and $Y \subset V(G)$ with $V(f) \subseteq Y$. Let $X^-$ and $Y^-$ be the corresponding sets in $G/V(f)$, where $z \in Y^-$. By Theorem~\ref{Thm: min counterexample},
    $7 \leq \vert \partial_G(X) \vert = \vert \partial_{G/V(f)}(X^-)\vert$ and 
    $7 \leq \vert \partial_G(Y) \vert = \vert \partial_{G/V(f)}(Y^-) \vert$. Hence,    
    $\vert \partial_{G'}(Z) \vert \geq 5$ for every odd set $Z \subset V(G')$ and therefore, 
    $G'$ is a projective planar 5-graph.
    Since $G'$ can also be obtained from $G$ by minor operations, $G'$ has no $P$-minor.
	
	By induction hypothesis, $G'$ is class 1. Let $e \in E_G(v,w)$ and $e_5,e_6\in f'$.
	It is easy to see, that any $5$-edge coloring of $G'$ extends to a 
	strong $e$-coloring of $G$, such that there are three perfect matchings which do not contain the edges $uw$ and $e_i$, $i \in \{4,5,6\}$. Further, the edges of $E_{G}(v,w)$ and $uv$ get all five colors.
    Thus, $val(f') \geq 3$, $d_G(f') \geq 6$, and $val(f'') \geq 3$, $d_G(f'') \geq 5$ by Lemma~\ref{Lemma: e-coloring Mate}.
\end{proof}

\begin{figure}[ht]%
	\centering
	
	\includegraphics[height=0.4\textwidth]{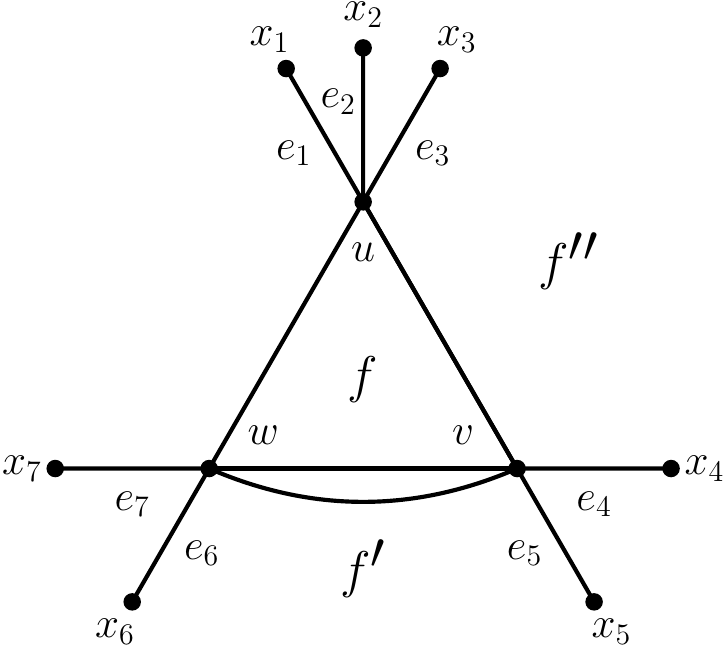}

	\caption{Lemma~\ref{Lemma: e-coloring fat triangle} if $\vert N_G(v) \vert = 4$.
    }\label{fig: e-coloring fat triangle}
\end{figure}

\begin{lem} \label{Lemma: diamond}
Let $f_1$ and $f_2$ be two adjacent 3-faces and $f \not \in \{f_1,f_2\}$. 
If $f$ is adjacent to $f_1$ or $f_2$, then $val(f) \geq 4$. 
\end{lem}

\begin{proof}
Let $f_1 = v_1v_2v_3$ and $f_2 = v_1v_3v_4$. By Lemma~\ref{Lemma: e-coloring fat triangle}, $f_1$ and $f_2$ are not heavy.
    Let $G'$ be the $v_1v_2v_3v_4$-swap of $G$.
    By Lemma~\ref{Lemma: swap}, $G'$ is a 
      projective planar $5$-graph with no $P$-minor. By the minimality of $G$
      it has a proper 5-edge coloring $c$.
      If $|c(E_{G'}(v_1,v_2))\cap c(E_{G'}(v_3,v_4))| \geq 1$, then $G$ has a proper 5-edge coloring, so
      we can assume $c(v_3v_4) = 5, c(\overline{v_1v_2}) = 2, c(v_1v_2) = 3, c(v_1v_3) = 4$, and $c(\overline{v_3 v_4}) = 1$.
      Then $M_1 = (c^{-1}(1) \backslash \{\overline{v_3v_4}\}) \cup \{v_1v_2, v_1v_4, v_2v_3\}$ is a $T$-join in $G$ ($T = V(G)$) with $v_1$ and $v_2$ being the only vertices of degree 3.
      Further, $M_2 = (c^{-1}(2)\backslash \{\overline{v_1v_2}\})\cup \{v_1v_2\}, M_3 = c^{-1}(3), M_4 = c^{-1}(4)$, and  $M_5 = c^{-1}(5)$ are perfect matchings of $G$.
      Thus, $M_1, \ldots, M_5$ is a strong $v_1v_2$-coloring of $G$.

    Let $f \in \{f_1', \dots,f_4'\}$, where $f_i'$ is adjacent to $f_1$ or $f_2$
    with $v_iv_{i+1} \in W(f_i)$.
    By Lemma~\ref{Lemma: e-coloring Mate}, there exists a mate $\partial_G(X_j)$ for every $j \in [5]$.
      Since $v_1v_3 \in M_4, v_3v_4 \in M_5, v_1v_4,v_2v_3 \in M_1$ and $v_1v_2 \cap M_l \neq \emptyset$ for $l \in \{1,2,3\}$, and $f$ contains at least one edge from
      $\partial_G(X_t)$ for each $t \in \{2,3,4,5\}$.
      The face sharing an edge of $M_j$ with $f_1'$ and $f_3'$ contains another edge of $M_j$, which is not incident to that edge.
      Thus, these faces are of size at least $4$.
    By symmetry, the same conclusion holds for $f_2'$ and $f_4'$.
\end{proof}

The following corollary is an immediate consequence of Lemma~\ref{Lemma: diamond}.
\begin{cor} \label{Cor: Diamond}
    A 3-face is adjacent to at least two $4^+$-faces.
\end{cor}

\begin{lem}\label{Lemma: 5-case 4-face 2-face}
    A $4$-face is adjacent to at most one $2$-face.
\end{lem}

\begin{proof}
    Let $f=v_1v_2v_3v_4$ be a $4$-face. Suppose to the contrary that $f$ is 
    adjacent to more than one $2$-face. 
    
    Case 1: $\vert E_G(v_1,v_2)\vert = \vert E_G(v_3,v_4) \vert = 2$.
    Let $G'$ be the $v_1v_2v_3v_4$-swap of $G$.
    By Lemma~\ref{Lemma: swap}, $G'$ is a 
    projective planar $5$-graph with no $P$-minor which by the minimality of $G$ has  
    a proper 5-edge coloring $c$.
    Note that $\vert E_{G'}(v_1,v_2)\vert = \vert E_{G'}(v_3,v_4) \vert = 3$, and hence, $c(E_{G'}(v_1,v_2)) \cap c(E_{G'}(v_3,v_4)) \not = \emptyset$.
    Therefore, $G$ has a proper 5-edge coloring, a contradiction. 

    Case 2: $\vert E_G(v_1,v_2) \vert = \vert E_G(v_2,v_3) \vert = 2$.
    Let $G'$ be the $v_1v_2v_3v_4$-swap of $G$. As in case 1, we deduce that 
    $G'$ is a $5$-graph with no $P$-minor and that $c$ is a proper 5-edge coloring 
    of $G'$.
    Since $G$ is class 2, it follows that 
    $c(E_{G'}(v_1,v_2)) \cap c(E_{G'}(v_3,v_4)) = \emptyset$. Hence, we can assume that $c(v_2v_3) \not \in c(E_{G'}(v_1,v_2))$. Consequently, 
    $c(E_G(v_1,v_2)) \cap c(E_{G'}(v_3,v_4))  \not = \emptyset$
    and thus, $G$ has a proper 5-edge coloring, a contradiction.
\end{proof}

\begin{lem}\label{Lemma: 5-case 4-face 3-face antipode 4+}
    Let $f$ be $4$-face and $M$ be a perfect matching of its facial walk. If one of the edges of $M$ is in a $3$-face, then the other edge of $M$ is in a $4^+$-face $f'$
    with $val(f') \geq 4$.
\end{lem}

\begin{proof}
    Let $f=v_1v_2v_3v_4$.
    Without loss of generality, suppose that $v_1v_2v'$ is a 3-face.
    Let $G'$ be the $v_1v_2v_3v_4$-swap of $G$.
    By Lemma~\ref{Lemma: swap}, $G'$ is a $5$-graph with no $P$-minor.
    By Lemma~\ref{Lemma: 5-case 4-face 2-face}, $G'$ is smaller than $G$ and thus, it
    has a proper 5-edge coloring $c$.
    Since $G$ is class 2,  we can assume $c(E_{G'}(v_1,v_2)) \cap c(E_{G'}(v_3,v_4)) = \emptyset$, say $c(E_{G'}(v_1,v_2)) = \{2,3\}$ and $c(E_{G'}(v_3,v_4)) = \{1,4\}$.
    Then $M_1 = (c^{-1}(1) \backslash \{\overline{v_3v_4}\}) \cup \{v_1v_2, v_1v_4, v_2v_3\}$ is a $T$-join in $G$ ($T = V(G)$) with $v_1$ and $v_2$ being the only vertices of degree 3.
    Further, $M_2 = (c^{-1}(2)\setminus \{\overline{v_1v_2}\}) \cup \{v_1v_2\}, M_3 = c^{-1}(3), M_4 = c^{-1}(4)$, and  $M_5 = c^{-1}(5)$ are perfect matchings of $G$.
    Thus, $M_1, \ldots, M_5$ is a strong $v_1v_2$-coloring of $G$.

    By Lemma~\ref{Lemma: e-coloring Mate}, there exists a mate $\partial_G(X_t)$ for every $t \in [5]$. 
    Since $v_1v_4,v_2v_3 \in M_i$, $E_G(v_3,v_4) \subseteq \partial_G(X_i)$ for 
    each $i \in \{2,3,4,5\}$. If $5 \not \in \{c(v_1v'), c(v_2v')\}$, then $\partial_G(X_j)$ does not exist for each $j \in \{2,3,5\}$, a contradiction.
    Thus, we suppose $c(v_1v') = 5$.
    Then $\partial_G(X_j)$ contains edge $v_1v'$ .
    Therefore, $\{ v_1v',v_1v_2, v_3v_4 \} \in \partial_G(X_j)$ and there is an edge of $M_j$ in $f'$.
    The face sharing an edge of $M_j$ with $f'$ contains another edge of $M_j$, which is not in incident to that edge.
    Thus, these faces are of size at least $4$.
    Since $f$ is a 4-face, it follows that $val(f') \geq 4$ and $d(f') \geq 4$.
\end{proof}

\begin{lem}\label{Lemma: 4-face shame}
    If a heavy 4-face $f=v_1v_2v_3v_4$ with $E_G(v_1,v_4)=2$ is adjacent to a 3-face $f'=v_1v_2v_5$, then $val(f'') \geq 3$ for the face $f''$ which shares the edge $v_2v_5$ with $f'$. 
\end{lem}

\begin{proof}
    By Lemma~\ref{Lemma: 5-case 4-face 2-face}, $\vert E_G(v_2,v_3)\vert = 1$. Let $G'$ be the
     $v_1v_2v_3v_4$-swap of $G$.
    By Lemma~\ref{Lemma: swap}, $G'$ is a $5$-graph with no $P$-minor.
    By the minimality of $G$, $G'$ has a proper 5-edge coloring $c$.
    If $|c(E_{G'}(v_1,v_2)) \cap c(E_{G'}(v_3,v_4))| \geq 1$, then $G$ has a proper 5-edge coloring so we can assume $c(E_{G'}(v_1,v_2)) = \{2,3\}, c(E_{G'}(v_3,v_4)) = \{1,5\}, c(v_1v_4) = 4, c(\overline{v_1v_2}) = 2$, and $c(\overline{v_3v_4}) = 1$.
    Then $M_1 = (c^{-1}(1) \backslash \{\overline{v_3v_4}\}) \cup \{v_1v_2, v_1v_4, v_2v_3\}$ is a $T$-join in $G$ ($T = V(G)$) with $v_1$ and $v_2$ being the only vertices of degree 3.
    Further, $M_2 = (c^{-1}(2)\backslash \{\overline{v_1v_2}\})\cup \{v_1v_2\}, M_3 = c^{-1}(3), M_4 = c^{-1}(4)$, and  $M_5 = c^{-1}(5)$ are perfect matchings of $G$.
    Thus, $M_1, \ldots, M_5$ is a strong $v_1v_2$-coloring of $G$.

    By Lemma~\ref{Lemma: e-coloring Mate}, there exists a mate $\partial_G(X_t)$ for every $t \in [5]$. 
    Since $v_1v_4,v_2v_3 \in M_1$, each $\partial_G(X_i)$ contains $E_G(v_3,v_4)$ for $i \in \{2,3,4,5\}$.
    It follows that each $M_i$ must contain either $v_1v_5$ or $v_2v_5$.
    Thus, $v_1v_5 \not\in M_4$ since $c(v_1v_5)\neq 4$. Therefore, $\partial_G(X_j)$ contains $v_2v_5$ for $j \in \{ 2,3,4 \}$.
    Then face $f''$ which shares edge $v_2v_5$ with $f'$ contains at least one edge of $\partial_G(X_j)$.
    The face sharing an edge of $M_j$ with $f''$ contains another edge of $M_j$, which is not incident to that edge.
    Thus, these are $4^+$-faces.
    So $val(f'') \geq 3$.
\end{proof}



\begin{lem}\label{Lemma: 5-case 5-face 2-face}
    A $5$-face is adjacent to at most two $2$-faces.
\end{lem}

\begin{proof}
    Let $f=v_1\dots v_5$ be a $5$-face.
    Suppose to the contrary that $f$ is adjacent to more than two $2$-faces. Then two of them are incident. 
    Without loss of generality, suppose $|E_G(v_1,v_2)| = |E_G(v_2,v_3)| = 2$.
    By Lemma~\ref{Lemma: three 2-edges in a row}, $|E_G(v_3,v_4)| = |E_G(v_1,v_5)| = 1$ and consequently, $|E_G(v_4,v_5)| = 2$.

    Let $G'$ be the graph obtained from $G$ by first contracting 
    $V'=\{v_1,v_2,v_3\}$ to a new vertex $v'$, deleting the edges $v_4v',v_5v'$, and adding an edge $\overline{v_4v_5}$ between $v_4,v_5$.

    Let $X,Y$ be odd sets in $G$ where $X \subseteq V(G) \setminus V'$ and $Y \subset V(G)$ with $V' \subseteq Y$. Let $X^-$ and $Y^-$ be the 
    corresponding sets in $G/V'$, where $v' \in Y^-$. By Theorem~\ref{Thm: min counterexample},
    $7 \leq \vert \partial_G(X) \vert = \vert \partial_{G/V'}(X^-)\vert$ and 
    $7 \leq \vert \partial_G(Y) \vert = \vert \partial_{G/V'}(Y^-) \vert$. Hence,    
    $\vert \partial_{G'}(Z) \vert \geq 5$ for every odd set $Z \subset V(G)$ and hence, $G'$ is a projective planar 5-graph.

Let $\{e_1,e_1', e_2, e_3,e_3'\} = \partial_{G'}(v')$, where we consider
the corresponding edge (to $e_i, e_i'$) in $G$ to be incident to $v_i$. 
Then $c(e_2) \not \in c(E_{G'}(v_4,v_5))$, otherwise $c$ can be extended to
a proper 5-edge coloring of $G$. 
It follows that either $\vert c(\{e_1,e'_1\}) \cap c(E_{G'}(v_4,v_5))\vert = 2$ or 
$\vert c(\{e_3,e'_3\}) \cap c(E_{G'}(v_4,v_5))\vert  = 2$,
    say $\vert c(\{e_1,e'_1\}) \cap c(E_{G'}(v_4,v_5))\vert = 2$.
    To be precise, let  $c(e_2) = 1, c(e_1) = 2, c(e'_1) = 3, c(e_3) = 4, c(e'_3) = 5$, and $c(E_{G'}(v_4,v_5)) = \{2,3,4\}$ where $c(\overline{v_4v_5}) = 2$, and $e\in E_{G}(v_4,v_5)$.
Extend $c$ to a partial $5$-edge coloring of $G$ as follows: 
    Set $c(E_{G}(v_1,v_2)) = \{ 4,5\}$ and $c(E_{G}(v_2,v_3)) = \{ 2,3\}$.
    Then $M_1 = c^{-1}(1) \cup \{v_1v_5, v_3v_4, e\}$ is a $T$-join in $G$ ($T = V(G)$) with $v_4$ and $v_5$ being the only vertices of degree 3.
    Further, $M_2 = (c^{-1}(2)\backslash \{\overline{v_4v_5}\})\cup \{e\}, M_3 = c^{-1}(3), M_4 = c^{-1}(4)$, and  $M_5 = c^{-1}(5)$ are perfect matchings of $G$.
    Thus, $M_1, \ldots, M_5$ is a strong $e$-coloring of $G$. But there is no mate for $M_2$,
    a contradiction to Lemma~\ref{Lemma: e-coloring Mate}. \end{proof}

\begin{lem}\label{Lemma: 5-case 5-face direct heavy triangle}
    If a $5$-face $f$ is adjacent to two $2$-faces and a heavy 3-face,
    then $f$ is not adjacent to another 2- or 3-face. 
\end{lem}

\begin{proof}
    Let $f= v_1\dots v_5$ be a 5-face which is adjacent to two 2-faces and a heavy 3-face.
    Suppose to the contrary that $f$ is adjacent to another 3-face. Obviously,
    the two $3$-faces cannot be adjacent. 

    Case 1: The two 2-faces adjacent to $f$ are incident to each other, say $\vert E_G(v_1,v_2)\vert = \vert E_G(v_2,v_3)\vert = 2$.
    
    First, if the two 3-faces share a vertex in $W(f)$, say
    $v_3v_4v_1'$ and $v_4v_5v_2'$ are the 3-faces, then
    by Lemma~\ref{Lemma: three 2-edges in a row}, $\vert E_G(v_3,v_1') \vert = 1$.
    By Lemma~\ref{Lemma: e-coloring fat triangle}, $\vert E_G(v_4,v_1') \vert = \vert E_G(v_4,v_2') \vert = \vert E_G(v_5,v_2') \vert = 1$.
    This contradicts the fact that $f$ is adjacent to a heavy 3-face.

    Second, if the two 3-faces do not share a vertex in $W(f)$, say 
    $v_3v_4v'$ and $v_1v_5v_6$ are the 3-faces, then
    by Lemma~\ref{Lemma: e-coloring fat triangle}, if $\vert E_G(v_5,v_6) \vert = 2$, then $val(f) \geq 3$, a contradiction.
    So we can assume $\vert E_G(v_1,v_6) \vert = 2$.

    Let $G'$ be the $v_1v_2v_3v_4v_5v_6$-swap of $G$.
    By Lemma~\ref{Lemma: swap}, $G'$ is a $5$-graph which is projective planar and $G'$ has no $P$-minor.
    By the minimality of $G$, $G'$ has a proper 5-edge coloring $c$.
    Without loss of generality, we may assume that $1\in c(E_{G'}(v_1,v_2))\cap c(E_{G'}(v_5,v_6))$ with $c(\overline{v_1v_2})=c(\overline{v_5v_6})=1$.
    Then, $1\notin c(E_{G'}(v_3,v_4))$, as otherwise $G$ is $5$-edge colorable. 
    Let $c(\overline{v_3v_4})=2$ and $c({v_3v_4})=3$.
    Then $M_1 = (c^{-1}(1)\backslash \{ \overline{v_1v_2}, \overline{v_5v_6} \} ) \cup \{ v_1v_6, v_2v_3, v_3v_4, v_4v_5 \}$ is a $T$-join in $G$ ($T = V(G)$) with $v_3$ and $v_4$ being the only vertices of degree 3.
    Furthermore, $M_2 = (c^{-1}(2)\backslash \{\overline{v_3v_4}\})\cup \{v_3v_4\}, M_3 = c^{-1}(3), M_4 = c^{-1}(4)$, and  $M_5 = c^{-1}(5)$ are perfect matchings of $G$.
    Thus, $M_1, \ldots, M_5$ is a strong $v_3v_4$-coloring of $G$.
    By Lemma~\ref{Lemma: e-coloring Mate}, there exists a mate $\partial_G(X_j)$ for every $j \in [5]$.
    As $c(v'v_3),c(v'v_4)\neq 2$ or $3$,
    we may assume then $2\in \{c(v_1v_5), c(v_5v_6)\}$ and $3\in c(E_{G}(v_1,v_2))$, which implies that there is no mate for $4$ and $5$.
    
    Case 2: The two 2-faces adjacent to $f$ are not incident to each other, say $\vert E_G(v_1,v_2)\vert = \vert E_G(v_3,v_4)\vert = 2$.
    As there are two 3-faces incident to $f$, we can assume that $v_1v_5v_6$ is a 3-face.

    Let $G'$ be the $v_1v_2v_3v_4v_5v_6$-swap of $G$.
    By Lemma~\ref{Lemma: swap}, $G'$ is a $5$-graph which is projective planar 
    and $G'$ has no $P$-minor.
    By the minimality of $G$, $G'$ has a proper 5-edge coloring $c$.

    Without loss of generality, we may assume that $1\in c(E_{G'}(v_1,v_2))\cap c(E_{G'}(v_5,v_6)$.
    Note that $c(E_{G'}(v_1, v_2))\cap c(E_{G'}(v_3, v_4))\cap c(E_{G'}(v_5, v_6))=\emptyset$, as otherwise $G$ is $5$-edge colorable.
    In particular, we can assume that $1\notin c(E_{G'}(v_3,v_4))$. 
    Let $e\in E_{G}(v_3,v_4)$.
    Then $c$ can be extended to a strong $e$-coloring of $G$ with $v_3$ and $v_4$ being incident to three edges of color $1$.
    Let $M_1, \ldots, M_5$ be the strong $e$-coloring.
    By Lemma~\ref{Lemma: e-coloring Mate}, there exists a mate $\partial_G(X_j)$ for every $j \in [5]$. 
    Then there are two colors $a,b\in \{2,3,4,5\}$, such that $a\in c(E_{G}(v_1,v_2))\cap c(E_{G}(v_3,v_4))$ and $b\in \{c(v_1v_5), c(v_5v_6)\}\cap c(E_{G}(v_3,v_4))$, which implies that there is no mate for the other two colors, a contradiction.
\end{proof}



\begin{lem}\label{Lemma: 6-face 2-faces}
    If a 6-face $f$ is adjacent to four 2-faces, then these four 2-faces 
    form a 2-face 4-path.
\end{lem}

\begin{proof}
    Let $f=v_1 \dots v_6$ be a 6-face.
    Suppose to the contrary, by Lemma~\ref{Lemma: three 2-edges in a row}, we can assume $|E_G(v_1,v_2)| = |E_G(v_2,v_3)| = |E_G(v_4,v_5)| = |E_G(v_5,v_6)| = 2$.

    Let $G'$ be the $v_1v_2v_3v_4v_5v_6$-swap of $G$.
    By Lemma~\ref{Lemma: swap}, $G'$ is a 5-graph with no $P$-minor.
    By the minimality of $G$, $G'$ has a proper 5-edge coloring $c$. 
    Without loss of generality, we may assume that $1\in c(E_{G'}(v_1,v_2))\cap c(E_{G'}(v_3,v_4))$, $c(\overline{v_1v_2})=c(\overline{v_3v_4})=1$, and $c(\overline{v_5v_6})=2$.
    In particular, $1\notin c(E_{G'}(v_5,v_6))$, as otherwise $G$ is $5$-edge colorable.
    Let $E_{G}(v_5,v_6)=\{e_1,e_2\}$.
    Then $M_1 = (c^{-1}(1)\backslash \{ \overline{v_1v_2}, \overline{v_3v_4} \} ) \cup \{ v_1v_6, v_2v_3, v_4v_5, e_1 \}$ is a $T$-join in $G$ ($T = V(G)$) with $v_5$ and $v_6$ being the only vertices of degree 3.
    Further, $M_2 = (c^{-1}(2)\backslash \{\overline{v_5v_6}\})\cup \{e_1\}, M_3 = c^{-1}(3), M_4 = c^{-1}(4)$, and  $M_5 = c^{-1}(5)$ are perfect matchings of $G$.
    Thus, $M_1, \ldots, M_5$ is a strong $e_1$-coloring of $G$.

    By Lemma~\ref{Lemma: e-coloring Mate}, there exists a mate $\partial_G(X_j)$ for every $j \in [5]$.
    There exist colors $a,b\in \{2,3,4,5\}$, such that $a\in c(E_{G}(v_5,v_6))\cap c(E_{G}(v_1,v_2))$ and $b\in c(E_{G}(v_5,v_6))\cap c(E_{G}(v_3,v_4))$, which implies that there is no mate for the other two colors, a contradiction.   
\end{proof}

\begin{lem}\label{Lemma: 6-face 2-faces incident}
    If a $6$-face is adjacent to three $2$-faces, then two of the $2$-faces are incident.
\end{lem}

\begin{proof}
    Let $f = v_1 \dots v_6$ be a 6-face and assume to the contrary that $|E_G(v_1,v_2)| = |E_G(v_3,v_4)| = |E_G(v_5,v_6)| = 2$. Let 
    $G'$ be the $v_1v_2v_3v_4v_5v_6$-swap of $G$. 
    We deduce a contradiction to our assumption as in the proof of Lemma~\ref{Lemma: 6-face 2-faces}.
\end{proof}




\subsection{Discharging procedure}

Assign to each face $f$ of $\Psi(G)$
the initial charge of $d(f) - \frac{10}{3}$, which is denoted by $\omega(f)$. 
As $G$ is a $5$-graph, Euler's formula implies that $|F(G)| = \frac{3}{5}|E(G)| + 1$, where $F(G)$ denotes the set of faces of $\Psi(G)$.
Therefore, the total amount of charge is 

\begin{equation} \label{Eq: total initial charge 5}
\sum_{f\in F(G)}\omega(f) = \frac{10}{3}(|F(G)| -1) - \frac{10}{3}|F(G)| = -\frac{10}{3}.
\end{equation}

To finish the proof, we show that charge can be redistributed such that 
every face has non-negative charge, in contradiction to equation~\ref{Eq: total initial charge 5}.

Next, we define the rules for the redistribution of charge and then we show
that every face has non-negative charge.

\subsubsection{Rules}

\begin{description}\setlength{\itemsep}{0mm}\setlength{\parsep}{0mm}
    \item[R2a] Each $2$-face adjacent to only $4^+$-faces receives charge $\frac{2}{3}$ from each adjacent face;
    \item[R2b] Each $2$-face adjacent to a $3$-face receives charge $\frac{1}{3}$ from the $3$-face and charge $1$ from the adjacent $4^+$-face;
    \item[R2c] Each $2$-face adjacent to a $2$-face receives charge $\frac{4}{3}$ from the adjacent $4^+$-face;
    \item[R3a] Each $3$-face adjacent to only $4^+$-faces which are not heavy $4$-faces receives charge $\frac{1}{9}$ from each adjacent face;
    \item[R3b] Each $3$-face adjacent to a $3$-face receives charge $\frac{1}{6}$ from each adjacent $4^+$-face;
    \item[R3c] Each $3$-face adjacent to a $2$-face receives charge $\frac{1}{3}$ from each adjacent $4^+$-face;
    \item[R3d] Each $3$-face which is adjacent to a heavy $4$-face receives charge $\frac{1}{6}$ from each adjacent $4^+$-face which is not heavy.
\end{description}

\begin{figure}[ht]
    \centering
    \begin{minipage}{0.32\textwidth}
    \centering
        \includegraphics[scale=0.6]{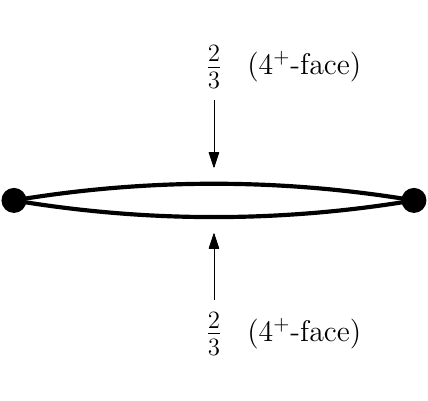}\caption*{R2a.}
    \end{minipage}
    \begin{minipage}{0.32\textwidth}
    \centering
        \includegraphics[scale=0.6]{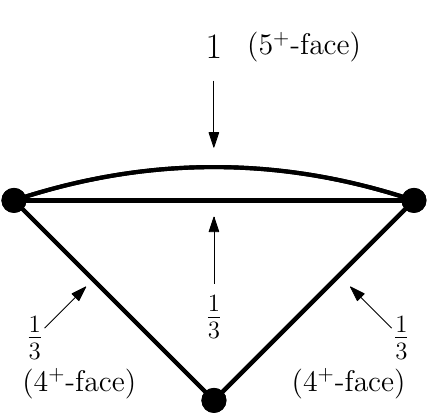}\caption*{R2b / R3c.}
    \end{minipage}
    \begin{minipage}{0.32\textwidth}
    \centering
        \includegraphics[scale=0.6]{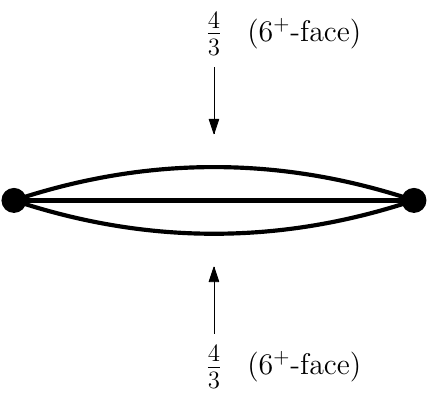}\caption*{R2c.}
    \end{minipage}\vspace{1em}

    \begin{minipage}{0.32\textwidth}
    \centering
        \includegraphics[scale=0.6]{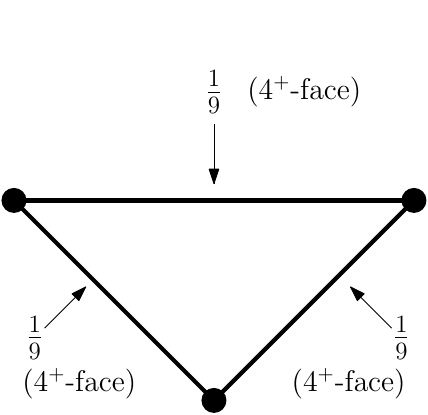}\caption*{R3a.}
    \end{minipage}
    \begin{minipage}{0.32\textwidth}
    \centering
        \includegraphics[scale=0.6]{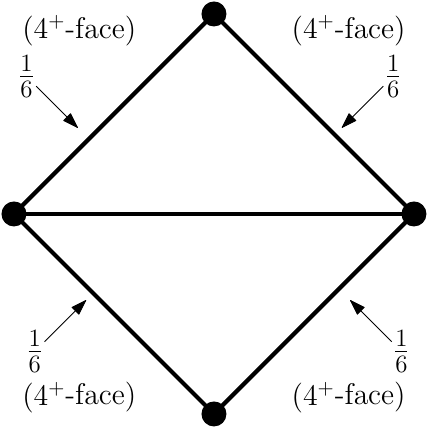}\caption*{R3b.}
    \end{minipage}
    \begin{minipage}{0.32\textwidth}
    \centering
        \includegraphics[scale=0.6]{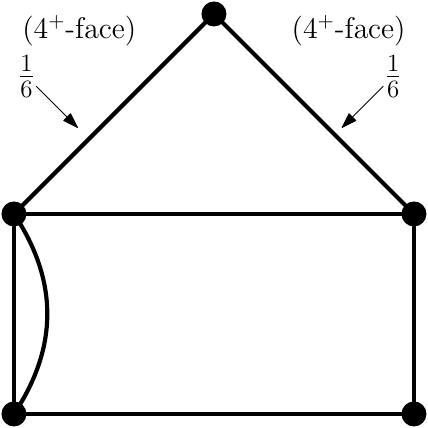}\caption*{R3d.}
    \end{minipage}
    \caption{Illustration of discharging rules.}
    \label{fig: Discharging Rules}
\end{figure}

Let $f$ be a $3^+$-face and $S$ be the set of the $2$- and $3$-faces which are adjacent to $f$. 
The amount of charge which is transferred  from $f$ to $f' \in S$ is denoted by $\tau(f,f')$. The total charge $\sum_{f' \in S}\tau(f,f')$ which is transferred from
$f$ to the faces of $S$ is denoted by $\tau(f,S)$. Note that, by
Observation~\ref{Obs: 2- and 3-faces}, each $f' \in S$ receives charge from $f$ at most once.

The final charge of a face $f$ after redistribution of charge  
is denoted by $\omega'(f)$.


\subsubsection{Non-negativity of faces}

In this section we prove stepwise that $\omega'(f) \geq 0$, for every face $f$ of $\Psi(G)$.
In stating the application of rules for a specific case, we refer to a worst-case scenario.

\subsubsection*{2-faces}

Let $f$ be a 2-face.
If $f$ is in a heavy 2-face, then by Lemma~\ref{Lemma: (r-2)-edge} and rule R2c, $\omega'(f) \geq 0$. Let $f_1, f_2$ be the two faces which are adjacent to $f$. 
If $f_1$ is a 3-face, then by Lemma~\ref{Lemma: three 2-edges in a row}, $f$ is not incident to two 2-faces.
Thus, $d(f_2) \geq 6$ by Lemma~\ref{Lemma: e-coloring fat triangle}  and 
therefore, $\omega'(f) \geq 0$ by rule R2b.
If $f_1, f_2$ are $4^+$-faces, then $\omega'(f) \geq 0$ by rule R2a.

\subsubsection*{3-faces}

Let $f$ be a 3-face which is adjacent to faces $f_1,f_2,f_3$. 
By Corollary~\ref{Cor: Diamond}, two of them, say $f_2, f_3$, are $4^+$-faces. 
If $f_1$ is a $4^+$-face, then by Lemma~\ref{Lemma: 4-face shame} at most one face adjacent to $f$ is heavy.
If none is heavy, then R3a applies.
If one is heavy, then rule R3d applies.
If $f_1$ is a 3-face, then by Corollary~\ref{Cor: Diamond} rule R3b applies.
If $f_1$ is a 2-face, then rule R2b and R2c apply. 
Thus $\omega'(f) \geq 0$ in any case.


\subsubsection*{4-faces}

Let $f$ be a 4-face.
By Lemma~\ref{Lemma: (r-2)-edge} and Lemma~\ref{Lemma: e-coloring fat triangle},  $f$ is not adjacent to a heavy $2$-face and not \adjtwo\ to a heavy $3$-face.
If $f$ is adjacent to a 2-face, then by Lemma~\ref{Lemma: 5-case 4-face 2-face}, there is at most one 2-face adjacent to $f$ and hence, $f$ is a heavy $4$-face.
Thus, rule R2a or R3d applies, and hence $\omega'(f) \geq 0$ in any case.
Now suppose that $f$ is not adjacent to a 2-face. 
If $f$ is adjacent to a 3-face, then by Lemma~\ref{Lemma: 5-case 4-face 3-face antipode 4+}, $f$ is adjacent to at most two 3-faces.
Therefore, $\omega'(f) \geq 0$ by rules R3a, R3b, R3c, and R3d.
If $f$ is adjacent to only $4^+$-faces, then no discharging rule applies 
and hence, $\omega'(f) \geq 0$.

\subsubsection*{5-faces}

Let $f$ be a 5-face.
By Lemma~\ref{Lemma: (r-2)-edge} and Lemma~\ref{Lemma: 5-case 5-face 2-face}, $f$ is not adjacent to a heavy $2$-face and $f$ is adjacent to at most two 2-faces.
And then by Lemma~\ref{Lemma: e-coloring fat triangle}, $f$ is not \adjtwo\ to a heavy $3$-face.
If $f$ is adjacent to two 2-faces and to a heavy $3$-face, then by Lemma~\ref{Lemma: 5-case 5-face direct heavy triangle}, $f$ is adjacent to no other $3$-face. 
Therefore, $\omega'(f) \geq 0$ by rules R2a and R3c.
Assume now $f$ is not adjacent to any heavy $3$-face.
If $f$ is adjacent to two 2-faces and three 3-faces, then by Lemma~\ref{Lemma: diamond} and Lemma~\ref{Lemma: 4-face shame}, rules R3b and R3d do not apply to the $3$-faces adjacent to $f$.
Therefore, $\omega'(f) \geq 0$ by rules R2a and R3a.
If $f$ is adjacent to two 2-faces and at most two 3-faces, then $\omega'(f) \geq 0$ by rules R2a and R3d.
If $f$ is adjacent to exactly one 2-face, then by Lemma~\ref{Lemma: e-coloring fat triangle} $f$ is adjacent to at most two heavy $3$-faces.
Therefore, $\omega'(f) \geq 0$ by rules R2a, R3b, and R3c.
If $f$ is not adjacent to a 2-face, then by rule R3c, $\tau(f,f') \leq \frac{1}{3}$ for all faces $f'$ adjacent to $f$ and consequently, $\omega'(f) \geq 0$.


\subsubsection*{6-faces}

Let $f$ be a 6-face.
If $f$ is adjacent to a heavy $2$-face, then by Lemma~\ref{Lemma: (r-2)-edge}, $val(f) \geq 5$
and therefore $\omega'(f) \geq 0$ by rule R2c.
If $f$ is \adjtwo\ to a heavy triangle, then by Lemma~\ref{Lemma: e-coloring fat triangle}, $val(f) \geq 3$. Since $r=5$, it follows with Lemma 
\ref{Lemma: diamond} that no two heavy triangles
to which $f$ is indirectly adjacent are incident. Hence, by
Lemma~\ref{Lemma: 6-face 2-faces incident}, $f$ is \adjtwo\ to at most two heavy triangles.
This implies that for the three faces $f_1,f_2,f_3$ which receive charge from $f$, $\tau(f,f_1) + \tau(f,f_2) + \tau(f,f_3) \leq \frac{8}{3}$ and
consequently, $\omega'(f) \geq 0$. 
Thus, it remains to consider the cases when $f$ is not indirectly adjacent to a heavy 2- or 3-face. \\
If $f$ is adjacent to 2-faces, then by Lemma~\ref{Lemma: three 2-edges in a row}, $f$ is adjacent to at most four 2-faces.
If $f$ is adjacent to four 2-faces, then by Lemma~\ref{Lemma: 6-face 2-faces} they form a 2-edge 4-path.
By Lemma~\ref{Lemma: three 2-edges in a row}, $val(f) \geq 2$ and hence, by rule R2a, $\omega'(f) \geq 0$.
If $f$ is adjacent to exactly three 2-faces, then by Lemma~\ref{Lemma: 6-face 2-faces incident} at least two of them must be incident.
If the three 2-faces form a 2-face 3-path, then by Lemma~\ref{Lemma: three 2-edges in a row}, $val(f) \geq 1$.
Then $\tau(f,f') \leq \frac{1}{3}$ for the other faces $f'$ which receives charge from $f$.
Thus, by rules R2a and either R3a, R3b, or R3d, $\omega'(f) \geq 0$.
If the three 2-faces do not form a 2-face 3-path, then let the three faces adjacent to $f$ which are not 2-faces be $f_1',f_2',f_3'$.
By Lemma~\ref{Lemma: three 2-edges in a row} and Lemma~\ref{Lemma: e-coloring fat triangle} at most two of them are heavy 3-faces, and if two
of them are heavy 3-faces, then the third one is a $4^+$-face. Thus, $\tau(f,f_1') + \tau(f,f_2') + \tau(f,f_3') \leq \frac{2}{3}$ in any case and consequently, $\omega'(f) \geq 0$.
If $f$ is adjacent to at most two 2-faces, then by rules R2a and either R3a or R3c $\omega'(f) \geq 0$.


\subsubsection*{$7^+$-faces}

Let $f$ be a $7^+$-face. For an edge $e \in W(f)$ let $f_e$ be the
other face which contains $e$. 
A subpath $P$ of $W(f)$ is \emph{discharging} if $f_e$ is a 2- or 3-face for each 
$e \in E(P)$ or $\vert V(P) \vert = 1$. 
Let $P_f$ be the set of inclusion-wise maximal discharging paths of $W(f)$,
$S_P = \{f_e \colon e \in E(P)\}$, and $S_f = \bigcup_{P \in P_f} S_P$.
Then, $\vert P_f \vert = \vert S_f \vert = val(f)$ and
$d(f) - val(f) = \sum_{P \in P_f} \vert E(P) \vert$.\\
Let $e_1$ and $e_2$ be two consecutive edges of $W(f)$.
If $f_{e_1}$ is a heavy 2-face or a heavy 3-face to which $f$ is indirectly adjacent, then by Theorem~\ref{Thm: min counterexample} and the discharging rules,  $\tau(f,f_{e_1}) + \tau(f,f_{e_2}) \leq \frac{5}{3}$. Thus, 
$\tau(f,f_{e_1}) + \tau(f,f_{e_2}) \leq \frac{5}{3}$ in any case. \\
Consequently for $P \in P_f$, $\tau(f,S_P) \leq \frac{5}{6} \vert E(P) \vert $ if $P$ 
has even length and 
$\tau(f,S_P) \leq \frac{5}{6} (\vert E(P) \vert -1) + \max\{\tau(f,f_e),
\tau(f,f_{e'})\}$,
if $P$ is of odd length with end edges $e$ and $e'$. 
Let $t = val(f) - c_t$ ($c_t \geq 0$) be the number of paths of odd length in $P_f$.\\
If $S_f$ contains a heavy 2-face, then
$\tau(f,S_P) \leq \frac{5}{6} (\vert E(P) \vert -1) + \frac{4}{3} =
\frac{5}{6} \vert E(P) \vert + \frac{1}{2}$ for each odd path $P \in P_f$.
Hence, $\omega'(f) = d(f) - \frac{10}{3} - \tau(f,S_f)
    \geq d(f) - \frac{10}{3} - (\frac{5}{6}(d(f) - val(f)) + \frac{t}{2})
    = \frac{1}{6}(d(f) + 2 val(f) + 3c_t - 20) \geq 0$, since 
    $d(f) + 2 val(f) + 3c_t \geq 20$ by 
Lemma~\ref{Lemma: (r-2)-edge}. Note that $d(f) \geq 10$ if $c_t=0$.\\
If $S_f$ contains no heavy 2-face, then 
$\tau(f,S_P) \leq \frac{5}{6} \vert E(P) \vert + \frac{1}{6}$ for each odd path $P \in P_f$.
As above, we deduce
$\omega'(f) \leq \frac{1}{6}(d(f) + 4 val(f) + c_t - 20)$.
If $val(f) \geq 4$ or $val(f) = 3$ and $c_t \geq 1$, then $\omega'(f) \geq 0$. If $val(f) = 3$ and $c_t=0$, then $d(f)$ is even ($\geq 8$) and hence, $\omega'(f) \geq 0$.

It remains to consider the case $val(f) \leq 2$. 
By Lemma~\ref{Lemma: three 2-edges in a row} and Lemma~\ref{Lemma: e-coloring fat triangle}, $f$ is not \adjtwo\ to any heavy $3$-face. Again by Lemma~\ref{Lemma: three 2-edges in a row}, $f$ is adjacent to at most $(d(f)-2)$ $2$-faces. 
If $f$ is adjacent to exactly $(d(f)-2)$ $2$-faces, then by Lemma~\ref{Lemma: three 2-edges in a row} $f$ is not adjacent to a $3$-face, and hence, $\omega'(f) \geq 
\frac{1}{3}(d(f) - 6) \geq 0$. 
If $f$ is adjacent to at most $(d(f)-3)$ $2$-faces, then, $\omega'(f) \geq 
\frac{1}{3}(d(f) - 7) \geq 0$ and therefore, $\omega'(f) \geq 0$ for all $7^+$-faces.

In conclusion,  every face of $\Psi(G)$ has non-negative charge. Consequently,
$0 \leq \sum_{f \in F(G)} \omega'(f) = \sum_{f \in F(G)} \omega(f) = - \frac{10}{3}$,
a contradiction, and Theorem~\ref{Thm: main 5-graph} is proved.

\section{Concluding remarks} \label{Sec: concluding remarks}

Let   $ H_1, \ldots, H_t $ be a sequence of graphs such that $V( H_i)\subseteq V(H_1) $ for each $ i\in\{2,\ldots,t\} $.  Denote by $H_1+E(H_2)+\cdots+E(H_t)$ the graph obtained from $H_1$ by adding a copy of every edge of $H_i$ for every $i\in\{2,\dots,t\}$. 
Let $\ca M$ be a finite multiset of perfect matchings of the Petersen graph $P$. 
The graph $P+\sum_{M\in\ca M}M$ is denoted by $P^{\ca M}$.

\begin{theo} [\cite{ma2025sets}] \label{theo:P+matching}
	Let $P$ be the Petersen graph and $H$ be a 1-regular graph on 
	$V(P)$ with edge set $M$. Then $P+M$ is class $2$ if and only if $M\subseteq E(P)$.
\end{theo}

\begin{cor} \label{cor: P+M}
Let $P$ be the Petersen graph and $\mathcal{M}$ be a finite multiset of 
edge sets of 1-regular graphs on $V(P)$. Then $P+\sum_{M\in\ca M}M$ is class 2
if and only if $M \subseteq E(P)$, for each $M \in \mathcal{M}$.
\end{cor}

Let $r \geq  3$ and $\ca M$ be a multiset of $r-3$ perfect matchings of the Petersen graph. A projective planar $r$-graph which is constructed from $P^{\ca M}$ by repeatedly building $r$-sums with planar $r$-graphs is called $P^{\ca M}${\emph-like}. 

By Corollary~\ref{cor: P+M}, every $P^{\ca M}$-like $r$-graph is class 2. 
Let $M$ be one element of $\ca M$, $\ca M' = \ca M \setminus M$ and $E$ be 
the edge set of a 1-regular graph on $V(P)$ which is not a perfect matching of $P$. Then, every 
projective planar $r$-graph which is obtained from $P^{\ca M'} + E$ by repeatedly building $r$-sums with planar $r$-edge colorable $r$-regular graphs is class 1 by 
Corollary~\ref{cor: P+M}. Obviously, they all have a spanning $P^{\ca M'}$-like
subgraph. 

The complete characterization of the projective planar $r$-graphs of class 2 for $r>3$ seems to be a hard problem. There might be other constructions 
of class 2 projective planar $r$-graphs than 
the aforementioned ones. For instance, consider the case $r=4$.
The line graph of a graph $G$ is denoted by $L(G)$. 
Kotzig \cite{Kotzig_1957} proved that a bridgeless cubic graph is class 2 if and only if 
its line graph is class 2.

\begin{prop} Let $M$ be a perfect matching of $P$. There are infinitely many projective planar 4-graphs of class 2, which are not $P^M$-like.
\end{prop}

\begin{proof}
    Let $G$ be a 3-connected $P$-like graph of order $4n$ ($n \geq 3$). Then $L(G)$ is a projective planar $4$-graph, which is class 2 by the result of Kotzig. Suppose that $L(G)$ is $P^M$-like. Since $V(L(G)) > 10$,
    it contains non-trivial tight cuts. However, $G$ is 3-connected and therefore, $L(G)$ does not contain such cuts, a contradiction. 
\end{proof}

The line graph $L(G)$ of a $P$-like graph $G$ of order $4n$ ($n\geq 3$) 
is a projective planar 4-graph. 
It has a perfect matching $M$ and $G' = L(G)-M$ 
is a 3-graph of class 2. Hence, by Theorem~\ref{Thm: Mohar_etal P-like}, $G'$ is $P$-like. That is, $L(G)$ can also be obtained from $G'$ by adding the edges of a 
1-regular graph on $V(G')$ to $G'$. 
For $r=4$, the class 1 and the class 2 graphs obtained by the constructions from the beginning of this section can also be obtained in such a way.

\section*{Acknowledgements}
This work is supported by the following grants and projects:
The third author is supported by NSFC (No. 12401471) and ZJNSFC (No. QN25A010023).

\bibliography{Lit}{}
\addcontentsline{toc}{section}{References}
\bibliographystyle{abbrv}
	
\end{document}